\documentclass[11pt]{amsart}
\usepackage{amsmath}
\usepackage{amssymb}
\usepackage{amsthm}
\usepackage{graphicx}
\usepackage{tikz}
\usepackage{xcolor}
\usepackage{hyperref}
\usepackage{cleveref}
\usepackage{url}
\usepackage{comment}
\usepackage{enumerate}
\allowdisplaybreaks

\usepackage[margin=1.2in]{geometry}

\def\E{\mathbb{E}}

\def\Int{\mathrm{Int}}

\newcommand{\abs}[1]{\left\vert#1\right\vert}
\def\R{\mathbb{R}}

\def\Z{\mathbb{Z}}

\def\eps{\varepsilon}

\newcommand{\norm}[1]{\|#1\|}

\def\1{\mathbf{1}}

\def\lam {\lambda}
\def\Lam{\Lambda}
\def\tce{t_c + \eps}
\def\tce2{t_c + \frac{\eps}{2}}

\newcommand{\ob}[1]{\left(#1\right )} 
\newcommand{\cb}[1]{\left[#1\right ]} 
\newcommand{\set}[1]{\left\{#1\right\}} 

\newcommand{\ve}{\varepsilon}

\numberwithin{equation}{section}
\theoremstyle{plain}

\newtheorem{theorem}{Theorem}[section]
\newtheorem{lemma}{Lemma}[section]

\newtheorem{defn}{Definition}[section]

\newtheorem{remark}{Remark}[section]

\begin{document}

\title[Correlation decay for hard spheres]{Correlation decay for hard spheres via Markov chains}

\date{\today}

\author{Tyler Helmuth}
\address{Durham University}
\email{tyler.helmuth@durham.ac.uk}
\author{Will Perkins}
\address{University of Illinois at Chicago}
\email{math@willperkins.org}
\author{Samantha Petti}
\address{Harvard University}
\email{spetti@fas.harvard.edu}

\maketitle

\begin{abstract}
 We improve upon all known lower bounds on the critical fugacity and
  critical density of the hard sphere model in dimensions two and
  higher.  As the dimension tends to infinity our improvements are by
  factors of $2$ and $1.7$, respectively.  We make these improvements
  by utilizing techniques from theoretical computer science to show
  that a certain Markov chain for sampling from the hard sphere model
  mixes rapidly at low enough fugacities.  We then prove an
  equivalence between optimal spatial and temporal mixing for hard
  spheres to deduce our results.
  \end{abstract}

\section{Introduction}
\label{secIntro}

For a fixed radius $r>0$, the \emph{hard sphere model} in a volume
$\Lambda\subset \R^{d}$ at fugacity
$\lam \ge0$ is a random point process $\mathbf X$ defined by conditioning a Poisson point process of
intensity $\lam$ on $\Lam$ on the event that the points are at
pairwise distance at least $2r$, or in other words, conditioning on
the event that the points $\mathbf X$ are the centers of a packing of
spheres of radius $r$. 
Conditioned on the number $k$ of centers, the distribution is uniform
over all sphere packings of $\Lam$ with $k$ spheres.   Note that by a  `sphere packing' we simply mean a configuration of non-overlapping spheres, not a `close packing' which implies maximality.

The hard sphere model is a simple but fundamental model of monatomic
gases. Its theoretical importance is in part due to the fact that it
conjecturally possesses a crystalline phase~\cite{alder1960studies,Blanc,hoover1968melting}.
Understanding the phase diagram of the model has presented a
significant challenge even at the level of physics~\cite{Krauth}, and
mathematical results are almost exclusively
restricted to understanding the low-density (small $\lam$) phase
(see~\cite{Richthammer} for a notable exception).  In particular, it
is an open mathematical problem to prove the existence of a phase
transition in the hard sphere model.  Not only is the model the most studied example of a broad class of Gibbs point processes, but it has played a starring
role in the development of Markov chain Monte Carlo methods since the
beginning: the Metropolis algorithm was first applied to the study of
the two-dimensional hard sphere
model~\cite{metropolis1953equation}. 

We will give a more precise definition of the hard sphere model below,
but for now we restrict our attention to aspects of the problem
directly relevant to our results. 
The reader unfamiliar with the model
may find~\cite{lowen2000fun} to be an inspiring introduction and
broader overview. Without loss of generality, it will be convenient to
choose the radius $r=r_{d}$ such that each sphere has
volume one and $\abs{\mathbf{X}}$, the number of spheres in the random packing, is also the volume covered by the packing.

The \emph{critical fugacity} $\lam_c(d)$ in dimension $d$ is the supremum
over $\lam$ such that the hard sphere model has a unique infinite
volume limit in the sense of van Hove, i.e., such that the set of weak
limit points of $\{\mu_{\Lam,\lam}\}_{\Lam}$ as $\Lam \to \R^d$ is a singleton set.  If $\lam_c(d) < \infty$, then the hard sphere model exhibits a \emph{phase transition} at fugacity $\lam_c$.   When
$d=1$, $\lambda_{c}(d)=\infty$, but it is not known in any higher dimension
whether or not $\lam_c(d)< \infty$. It is believed that
$\lambda_{c}(d)$ is finite in dimension $3$ (and in some or all
dimensions $d\geq 4$), while the case $d=2$ is subtle and remains an
active area of investigation even in the physics
literature~\cite{Krauth,thorneywork2017two}.

Proving a lower bound on $\lam_c(d)$ amounts to proving the absence of
a phase transition for $\lam$ in an interval on the real line.
Developing new and more powerful methods for proving the absence of a
phase transition has been a central theme of statistical mechanics
(e.g. \cite{yang1952statistical,dobrushin1970definition,van1993uniqueness,BrydgesSC}),
not only due to the interest in understanding phase diagrams but also
because of a broad equivalence between the absence of phase
transitions and important probabilistic and dynamical properties of
finite and infinite systems~\cite{dobrushin1985completely}.  In
particular, the absence of a phase transition is related to mixing
properties of natural dynamics to sample from finite volume Gibbs
measures, i.e., 
the thermodynamics properties
are connected to the performance of natural algorithms to sample from
finite volume systems.  While this connection has been made
  precise in the setting of lattice
  systems~\cite{stroock1992logarithmic,dyer2004mixing,weitz2005combinatorial},
  it has only been shown under more restrictive conditions in the setting
  of Gibbs point processes like the hard sphere model (see the
  discussion below in Section~\ref{sec:prev}).

  Our main result is an improved lower bound on the critical fugacity
  $\lam_c(d)$ that we prove by developing an unrestricted equivalence
  between finite volume Markov-chain mixing properties and 
  correlation decay properties.

\begin{theorem}
\label{thmGenThm}
For all $d \ge 2$,  $\lam_c(d) \ge 2^{1-d}$.
\end{theorem}

The \emph{density} of the hard sphere model on $\R^{d}$
at fugacity $\lam$ is
\begin{equation}
  \label{eqDensity}
  \rho(\lam) = \liminf_{n \to \infty}  \frac{1}{n} \E_{Q_n, \lam} |\mathbf X |\,,
\end{equation}
where $Q_n$ is the $d$-dimensional cube of volume $n$ centered at the
origin, and the expectation is with respect to the hard sphere model on
$Q_n$ at fugacity $\lam$.  The use of liminf in \eqref{eqDensity} is
necessary as \emph{a priori} the limit is only known to exist for
Lebesgue-a.e.\ value of $\lambda$.  The \emph{critical density} $\rho_c(d)$ of the hard sphere model is
$\rho(\lam_c(d))$ (or $\lim_{\lam \to \infty} \rho(\lam)$ if
$\lam_c = \infty$).  That is, $\rho_{c}(d)$ is the limiting expected
packing density of the hard sphere model at the critical fugacity
$\lam_c(d)$.  By making use of Theorem~\ref{thmGenThm} we can obtain
an improved lower bound on the critical density.

\begin{theorem}
  \label{thmDensityBounds}
  For all $d \ge 2$, $\rho_c(d) \ge \frac{2}{3 \cdot 2^d}$.  As the
  dimension $d$ tends to infinity 
  we have $\rho_c(d) \ge (.8526 + o_d(1))2^{-d}$.
\end{theorem}

Before discussing how these results improve upon all previously
  known bounds, 
we briefly outline the proof of Theorems~\ref{thmGenThm}
and~\ref{thmDensityBounds}. At a high level this is done by adapting
and combining three ingredients from the study of algorithms,
probability theory, and combinatorics:
\begin{enumerate}
\item We analyze a Markov chain for sampling hard sphere
  configurations in a finite volume.  By using techniques from
  theoretical computer science, namely path coupling with an optimized
  metric (e.g.~\cite{vigoda2001note}), we show that this Markov chain
  mixes rapidly at small enough fugacity. The conclusion is
  Theorem~\ref{mixing} below. Our main contribution here is to
  rigorously implement the idea that the hard sphere model behaves
  like a hard-core lattice gas on a finite graph with many triangles.
\item We establish a continuous analogue of the equivalence of spatial
  and temporal mixing from lattice spin systems
  (e.g.~\cite{stroock1992logarithmic,dyer2004mixing}) to deduce
  exponential decay of correlations 
  from our fast mixing
  results.  For applications to bounds on the critical fugacity of the hard sphere model
    our main result here is Theorem~\ref{thmSingleTempToSpatial}. We
    also prove a full equivalence between spatial and temporal mixing,
    see Theorems~\ref{thmHeatBathTemptoSpatial}
    and~\ref{thm:heatbathConverse}. Our main contribution here
      is to prove an equivalence result for hard spheres that does not
      rely on the convergence of a cluster expansion; the importance
      of this is discussed below in Section~\ref{sec:prev}.
\item We achieve the bounds on the critical density $\rho_{c}$ in
  Theorem~\ref{thmDensityBounds} by applying non-trivial lower bounds
  on the expected packing density of the hard sphere
  model~\cite{jenssen2019hard}.
\end{enumerate}
After the presentation of some preliminaries in
  Section~\ref{sec:hard-spheres-with}, we outline the first two of
  these steps in more detail in Sections \ref{sec:SSM-I} and
  \ref{sec:OTM-I}. Before this we compare our results
  with those previously obtained in the literature.

\subsection{Previous results}
\label{sec:prev}
Historically, the main approach  
to proving the absence of a phase transition at low densities in the
hard sphere model is to use the cluster expansion. This is a
  convergent power series expansion for the pressure of the hard
  sphere gas.  The classical bound states that the cluster expansion converges for
all complex $\lam$ with $|\lam| \le e^{-1}2^{-d}$, and thus
$\lam_c(d) > e^{-1} 2^{-d}$.  For a given dimension it may be
  possible to improve upon this result, e.g., in two dimensions, 
Fern\'andez, Procacci, and Scoppola~\cite{fernandez2007analyticity},
proved the cluster expansion converges for
$|\lam| \le .1277$.  However, one does not expect to be able to
improve the constant $e^{-1}$ as the dimension $d$ tends to infinity: the singularity determining the radius of convergence of the cluster expansion is known
  to lie on the negative real axis, and there is a compelling (but
  non-rigorous) argument that this singularity is at
  $-e^{-1}2^{-d}(1+o_d(1))$, see~\cite{frisch1985classical}. 
    It is  known rigorously that the cluster expansion cannot
    have a radius of convergence greater than $2^{-d}$, see~\cite[Remark~3.7]{FernandezNguyen}.

To avoid the negative axis singularity a natural idea is to
  use techniques that do not require analyticity, e.g., 
  probabilistic techniques which concern only positive fugacities
  $\lam$. One approach in this direction was taken by
Hofer-Temmel~\cite{christoph2019disagreement}, who used disagreement
percolation~\cite{van1993uniqueness} and known bounds on the critical
intensity of $d$-dimensional Poisson--Boolean percolation to prove lower
bounds on the critical fugacity of the hard sphere model.  In
dimension $2$, his 
bound is $\lam_c(2) > .1367$. Hofer-Temmel's method and a bound based
on the non-rigorous `high-confidence' results
of~\cite{balister2005continuum} for Poisson-Boolean percolation gives
$\lam_c(2) >.28175$~\cite{christoph2019disagreement}.  The asymptotics
of the critical intensity of Poisson--Boolean percolation as
$d\to\infty$ are known, and this gives a bound of
$\lam_c(d) \ge (1+o(1))2^{-d}$, improving upon the cluster expansion
bound by a factor $e$.  Our probabilistic approach makes use
  of Markov chain mixing times, a well-honed tool, instead of
  disagreement percolation. For comparison, 
our bound $\lam_{c}(d)\geq 2^{1-d}$ is an
improvement of a factor $2$ as $d\to\infty$, and of more than $3$
compared to the rigorous results in dimension $2$.

Finally, there has been work on
developing exact sampling algorithms for the hard sphere model. 
Guo and Jerrum~\cite{jerrum2019perfect} showed that a partial rejection sampling 
algorithm is efficient in dimension $2$ for $\lam \le .21027$ and
Wellens improved this bound to
$\lam \le .2344$~\cite{wellens2018note}. For an approach combining
coupling from the past and rejection sampling, see~\cite{Wilson}. 

While it appears we are the first to use Markov chains to
  estimate the critical fugacity of the hard sphere model, there
have been previous works that obtain bounds on the critical density by
showing that certain Markov chains for sampling a configuration of
hard spheres mix rapidly. To lower bound the critical density these
chains make use of the canonical ensemble, meaning the configurations
consist of a fixed finite number of spheres in a finite volume.
Results of this type include Kannan, Mahoney, and Montenegro who
showed that a simple Markov chain for the canonical ensemble exhibits
rapid mixing for densities $\rho < 2^{-1-d}$~\cite{kannan2003rapid},
and Hayes and Moore who used an optimized metric to show that in
dimension $2$ this same Markov chain mixes rapidly at densities
$\rho < .154$~\cite{hayes2014lower}. The Markov chain studied in these
papers 
moves spheres in a non-local way. Dynamics involving only local moves
have been investigated by Diaconis, Lebeau and Michel as an
application of a more general geometric
framework~\cite{DiaconisLebeauMichel}; these local dynamics are
restricted to vanishing densities due to the existence of jammed
configurations of arbitrarily low density, see~\cite{Kahle}.

To convert bounds on the critical fugacity
to bounds 
on the critical density, we use the  
bound $\rho(\lam) > \frac{\lam}{1 + 2^d \lam}$ where $\rho(\lam)$ is the
packing density of the hard sphere model at fugacity $\lam$ (see
Section~\ref{secDensity}).  Surprisingly, the estimates on $\rho_c$ that result
  from combining this estimate with  the high-dimensional bounds
from~\cite{christoph2019disagreement}
coincide with the high-dimensional bounds on $\rho_c$
  from~\cite{kannan2003rapid}.  Our results improve upon this by a factor of $4/3$. Our
  stronger result as $d\to\infty$ in Theorem~\ref{thmDensityBounds},
  which is an improvement of roughly $1.7$, makes use of a better
  estimate for $\rho(\lam)$, see Section~\ref{secDensity}. Our results are
  also the best known in $d=2$.

As mentioned above, 
our argument showing that rapid mixing implies exponential decay of
correlations is based on the argument given in~\cite{dyer2004mixing}
for discrete spin systems on graphs. Previously there has been work
relating mixing times and correlation decay for continuous-time
birth-death chains for continuum particle systems with soft two-body
potentials~\cite{BertiniCancriniCesi}. Later works allowed for hard
core potentials~\cite{Wu,BoudouEtAl}, but  all of these results apply only in the low
density regime, i.e., within the domain of convergence of the cluster
expansion. For the reasons discussed above, it is essential for
  us to have a result that does not rely on the convergence of the
  cluster expansion. We achieve this by using combinatorial techniques in our proof of equivalence to avoid 
a low density hypothesis.

\subsection{Future directions}
\label{sec:future-directions}

One interesting benchmark for further progress on determining the
uniqueness phase of the hard sphere model would be to obtain
uniqueness for all $\rho\leq 2^{-d}$, the point at which the system no
longer trivially (by a union bound) contains free volume. Passing this
threshold appears to require new ideas. Another tool from computer
science that may be applicable to the hard sphere model is Weitz's
correlation decay method~\cite{Weitz}, although some adaptation will
be necessary to deal with the continuous nature of space for the hard
sphere model.\footnote{Subsequent to the posting of this article to
  the arXiv, one possible adaptation of Weitz's method to continuum particle
  systems has been found~\cite{michelen2020analyticity}.} 

 While it should be straightforward to generalize our methods to
    more hard shapes that are not spheres, the adaptation to more
    general (soft) two-body potentials is less clear. A generalization
  that covers stable two-body potentials (the setting of, e.g.,~\cite{penrose1963convergence,ruelle1963correlation}) would be interesting.

It is also worth remarking that the step in our arguments of obtaining
density bounds from fugacity 
bounds is a challenging and interesting problem in itself. At very low fugacities one can use the cluster
expansion to write a convergent power series formula for the
density. Since, however, we are interested in going beyond the radius
of convergence of the cluster expansion, this tool is not
available. Developing alternative approaches to estimating the
density, whether by analytic continuation, direct study of the Virial
series, or other means, would be very interesting.

Lastly, there is of course the long-standing open problem of
proving the existence of a phase transition for the hard sphere model.

\section{Spatial and temporal mixing}

In this section we define the hard sphere model with boundary
conditions, define the notions of strong spatial mixing and optimal
temporal mixing, and reformulate our main results in terms of
  these notions.

\subsection{Hard spheres with boundary conditions}
\label{sec:hard-spheres-with}
We begin by formally defining the hard sphere model in a bounded
measurable volume $\Lam \subset \R^d$.  Recall that we write
$r = r_{d}$ for the radius of a sphere of volume $1$ in $\R^d$. Let
$$\Lam_{\Int}= \{ x \in \Lam :\mathrm{dist}(x,\Lam^{c}) \ge r \}.$$  
The hard sphere model on volume
$\Lam$ at fugacity $\lam \ge 0$ with \textit{free boundary conditions} is a
Poisson point process of intensity $\lam$ on $\Lam_{\Int}$ conditioned
on the event that all points are at pairwise distance at least $2r$.
In words, the hard sphere model arises by conditioning on the event
that the points form the centers of a sphere packing in $\Lam$ with
spheres of volume $1$; we recall a sphere packing in a set $A$ is any
collection of pairwise disjoint open spheres that are entirely
contained in $A$.  Explicitly, the normalizing constant $Z_{\Lam}(\lam)$ is
\begin{equation}
  \label{eq:Z}
  Z_{\Lam}(\lam) = \sum_{k\geq 0} \int_{\Lam_{\Int}^{k}}
  \frac{\lam^{k}}{k!} \prod_{1\leq i<j\leq k} 1_{\norm{x_{i}-x_{j}}\geq
    2r} \prod_{i=1}^{k} dx_{i}
\end{equation}
where $dx_{i}$ is Lebesgue measure. We will denote the law of
$\mathbf X$ by $\mu_\Lam$ (the dependence on $\lam$ will be
suppressed).  The density of $\mu_\Lam$ on  $\Lam_{\Int}^k$ with respect to Lebesgue measure is given by the integrand
of~\eqref{eq:Z} divided by the partition function   $Z_{\Lam}(\lam)$.  Note that the requirement that spheres lie entirely
within $\Lam$ instead of just requiring the centers to lie in $\Lam$
makes no difference in the infinite volume limit, but it does have a
regularizing effect in finite volume.

We will also be interested in the hard sphere model with boundary
conditions $\tau$. More precisely, we define
$\tau \subseteq \Lam_{\Int}$ as a set of forbidden locations for
centers. The hard sphere
model on a volume $\Lam$ at fugacity $\lambda \ge 0$ with boundary conditions
$\tau$ is a Poisson point process of intensity $\lambda$ on
$\Lam_{\Int} \setminus \tau$ conditioned on the event that all points are at pairwise distance
at least $2r$. One possibility is that $\tau$ represents the volume
blocked by a set of permanently fixed spheres: if $Y$ is a set of
centers and
$\tau = \Lam_{\Int} \cap \left(\cup_{y \in Y} B_{2r}(y)\right)$, then
$\Lam_{\Int} \setminus \tau$ is the set of locations for centers that
do not overlap with spheres defined by the centers in $Y$. Note $\tau$
need not have this form. The law of the hard sphere model on $\Lam$
with boundary condition $\tau$ will be denoted by $\mu_\Lam^\tau$.

\subsection{Strong spatial mixing}
\label{sec:SSM-I}

Let $\Omega_{\Lam}$ be the set of all 
\emph{configurations} for the hard sphere model on $\Lam$, that is,
the set of all finite point sets in $\Lam_{\Int}$ whose pairwise
distance is at least $2r$. Similarly, let $\Omega_\Lam^\tau$ be the set of
configurations for the hard sphere model on $\Lam$ with boundary
conditions $\tau$.  In particular,
$\Omega_{\Lam } = \Omega_\Lam^\emptyset$.

For two probability measures $\mu_1$ and $\mu_2$ on $\Omega_{\Lam}$ we
let $\| \mu_{1}-\mu_{2}\| = \| \mu_1- \mu_2\|_{TV}$ denote their total
variation distance. For $\Lam' \subseteq \Lam$ probability measures on
  $\Omega_{\Lam}$ induce probability measures on $\Omega_{\Lam'}$ by retaining only the points in $\Lam'$. To control the resulting measures, we let
$\| \mu_1 -\mu_2\|_{\Lam'}$ denote the total variation distance
between the pushforward of $\mu_1$ and $\mu_2$ to measures on configurations
in $\Lam'$ under the projection map from $\Lam$ to $\Lam'$. 
 In particular, if $|\Lam'| <1$, then the only valid configuration is the
empty set of centers and so $\|\mu_1 - \mu_2\|_{\Lam'} =0$.  For  $\Lam \subset \R^d$ we denote its volume by $|\Lam|$.

We can now define the strong spatial mixing property.

\begin{defn}
  The hard sphere model at fugacity $\lam$ exhibits \emph{strong
    spatial mixing} (SSM) on $\R^d$ if there exist $\alpha, \beta >0$
  such that for all compact measurable 
  $\Lam' \subseteq \Lam \subset \R^d$ and any pair of boundary
  conditions $\tau$ and $\tau'$,
  \begin{equation} 
    \norm{\mu_{\Lam}^{\tau} - \mu_{\Lam}^{\tau'}}_{\Lam'}
    \le \beta \abs{\Lam'} \exp ( - \alpha \cdot \mathrm{dist}(\tau
    \triangle \tau', \Lam')).
  \end{equation}
\end{defn}

We define the strong spatial mixing threshold
on $\R^d$ as
 \begin{equation}
\lam_{\mathrm{SSM}}(d) = \sup \{ \lam : \text{ SSM holds for } \lam' < \lam \} \,.  \label{eq:SSM}
\end{equation}
It is well-known that a much weaker spatial mixing condition implies uniqueness of infinite volume
Gibbs measures (e.g.~\cite{ruelle1999statistical,dobrushin1985completely}), and so $\lam_c (d) \ge \lam_{\mathrm{SSM}}(d)$.  The
inequality can in principle be strict; for example, it is expected
that  $\lam_{\mathrm{SSM}}(2) < \lam_c(2)$.

\subsection{Optimal temporal mixing}
\label{sec:OTM-I}

Consider the following Markov chain on $\Omega_\Lam^\tau$, called the
\emph{single-center dynamics}.  Given a configuration
$X_t \in \Omega_\Lam^\tau$, form $X_{t+1}$ as follows:
\begin{enumerate}
  \item Pick $x \in \Lam$ uniformly at random.
  \item With probability $1/ (1+\lambda)$, remove any $y \in X_t$ with
    $\mathrm{dist}(x,y)< r$; that is, let
    $X_{t+1} = X_t \setminus B_{r}(x)$.
  \item With probability $\lambda/ (1+\lambda)$, attempt to add a
    center at $x$.  That is, let $X' = X_t\cup \{x\}$. If
    $X' \in \Omega_\Lam^\tau$, then set $X_{t+1} = X'$; if not, then
    set $X_{t+1}= X_t$.
\end{enumerate}
We show in Lemma~\ref{lem:stat} below that the stationary distribution of this Markov chain is indeed $\mu_{\Lam}^{\tau}$.

Following~\cite{dyer2004mixing}, our notion of optimal temporal mixing for Markov
chains in the next definition is essentially  
$O(n \log n)$ mixing for all regions $\Lam$ of volume $n$ and all boundary conditions.

\begin{defn}
  Let $n = |\Lam|$.  The single-center dynamics for the hard sphere
  model on $\R^d$ has \emph{optimal temporal mixing} at fugacity
  $\lam$ if there exist $b,c >0$ so that for any compact measurable
  $\Lam \subset \R^d$, any boundary condition $\tau$, any $s>0$, and
  any two instances $(X_{t})$ and $(Y_{t})$ of the single-center
  dynamics on $\Omega_\Lam^\tau$,
  \begin{equation}
  \label{eqOptTempDef}
    \norm{X_{\lfloor sn \rfloor} - Y_{\lfloor s n \rfloor }}_{TV} \le b n e^{- cs},
  \end{equation}
  where, by an abuse of notation, the left-hand side means the total variation
    distance between the laws of $X_{\lfloor sn\rfloor}$ and
    $Y_{\lfloor sn \rfloor}$.
\end{defn}

\subsection{New results}
\label{sec:NR}

Using the technique of coupling with an optimized metric from Vigoda's
work on the discrete hard-core model on bounded-degree
graphs~\cite{vigoda2001note}, we establish optimal temporal mixing of
the single-center dynamics for fugacities $\lam < 2^{1-d}$.

\begin{theorem}
	\label{thmMCmixing}
	For all $d \ge 2$ and all $\lam < 2^{1-d}$, the single-center
        dynamics for the hard sphere model on $\R^d$ exhibits optimal
        temporal mixing.
\end{theorem}
We then prove that optimal temporal mixing of the single-center
dynamics implies strong spatial mixing. 
\begin{theorem}
  \label{thmSingleTempToSpatial}
  Fix $d \ge 2, \lam >0$. 
  If the single-center dynamics has optimal temporal mixing on $\R^d$, then the hard sphere model on $\R^d$ exhibits strong
  spatial mixing.
\end{theorem}
Together these theorems imply
Theorem~\ref{thmGenThm}. 
\begin{proof}[Proof of Theorem~\ref{thmGenThm}]
  Theorems~\ref{thmMCmixing} and~\ref{thmSingleTempToSpatial} together
  immediately imply that
  $\lam_c(d)\ge \lam_{\mathrm{SSM}}(d) \ge2^{1-d}$, the first
  inequality by the remark following~\eqref{eq:SSM}.
\end{proof}

The proof of Theorem~\ref{thmSingleTempToSpatial} does not use anything
specific about the single-center dynamics except that it performs
updates within a randomly chosen ball of bounded radius.  Another
Markov chain with this property is the \textit{heat-bath dynamics}
with update radius $L>0$. To define this chain, recall the
$\ell$-parallel set $A^{(\ell)}$ of $A\subset\R^{d}$ is
\begin{equation}
A ^{(\ell)} = \{ x \in \R^d : \mathrm{dist}(x,A) \le \ell \}.
\end{equation}
In particular, given our definition of $\Lam_{\mathrm{Int}}$ above, we have $\Lam = \Lam_{\mathrm{Int}}^{(r)}$.
  To make one step of the heat-bath dynamics we pick a point
$x \in \Lam_{\mathrm{Int}}^{(L)}$ uniformly at random and then resample the centers
in $B_L(x)$ subject to the boundary conditions induced by the other
centers in the current configuration and $\tau$.  
Optimal temporal mixing for the heat-bath dynamics also implies strong spatial mixing.
\begin{theorem}
\label{thmHeatBathTemptoSpatial}
Fix $d \ge 2, \lam >0, L>0$.  If the heat-bath dynamics with update
radius $L$ has optimal temporal mixing on $\R^d$, then the hard sphere
model on $\R^d$ exhibits strong spatial mixing.
\end{theorem}
The proof of this theorem is essentially identical to that of
  Theorem~\ref{thmSingleTempToSpatial} (see
  Section~\ref{sec:preparatory-lemmas}), and hence will be omitted.
We also prove a converse to
Theorem~\ref{thmHeatBathTemptoSpatial}: that strong spatial mixing
implies that the heat-bath dynamics exhibit optimal temporal mixing,
provided the update radius is sufficiently large (we define optimal
temporal mixing for the heat-bath dynamics just as for the
single-center dynamics).
\begin{theorem}
  \label{thm:heatbathConverse}
  Fix $d \ge 2, \lam >0$.  If the hard sphere model on $\R^d$ exhibits
  strong spatial mixing, then there is an $L_{0}>0$ such that for
  $L\geq L_{0}$ the heat-bath dynamics with update radius $L$ exhibits
  optimal temporal mixing.
\end{theorem}

\subsection{Notation and conventions}
\label{sec:notation-conventions}

We briefly collect some frequently used concepts. $B_{\ell}(x)$
denotes the open ball of radius $\ell$ centered at $x\in \R^{d}$, and
$V_{\ell} = |B_{\ell}(x)|$ will denote the volume of this set. In
particular, $V_{r}=1$. More generally, $|A|$ will denote the Lebesgue
measure of $A\subset \R^{d}$. For $\Lam\subset \R^{d}$ the
\emph{$\ell$-parallel set} $\Lam^{(\ell)}$ of $\Lam$ is
$\{x\in \R^{d} : d(x,\Lam)\leq \ell \}$.  By an abuse of notation, if $B$
is a finite set, we will write $|B|$ for the cardinality of $B$.

\section{A rapidly mixing Markov chain for the hard sphere model}
\label{secMarkov}

In this section we prove that the single-center dynamics for the hard
sphere model at fugacities $\lambda < 2^{1-d}$ mixes rapidly. We begin
by reviewing Markov chain mixing.

\subsection{Markov chain mixing}
\label{sec:mcmix}
Let $\Omega$ denote the state space of a discrete time Markov
chain. Let $p^{(0)}$ be the initial probability distribution on
$\Omega$, and let $p^{(t)}$ be the distribution after $t$ steps of the
Markov chain. Suppose the chain has a unique stationary distribution
$\mu$.  The mixing time of the chain is a worst-case estimate for the
number of steps it takes the Markov chain to approach
stationarity. More precisely,
\begin{defn} The mixing time of a Markov chain is
  \begin{equation}
    t_{\text{mix}}(\ve)=\sup_{p^{(0)} \in \mathcal{P}}\min \set{  t : \Vert p^{(t)} -\mu\Vert_{TV} \leq \ve}
  \end{equation}
  where $\mathcal{P}$ denotes the set of probability distributions on
  $\Omega$.
\end{defn}  

A common approach to bounding the mixing time of a Markov chain
is to construct a coupling. For our purposes, a coupling of two Markov
chains $(X_{t})_{t\geq 0}$ and $(Y_{t})_{t\geq 0}$ on $\Omega$ is a
stochastic process $(X_t, Y_t)_{t\geq 0}$ with values in $\Omega \times \Omega$
such that the marginals $(X_t)_{t\geq 0}$ and $(Y_t)_{t\geq 0}$ are
faithful copies of the Markov chains, and $X_{t+1}=Y_{t+1}$ whenever
$X_t= Y_t $.

The path coupling theorem of Bubley and Dyer~\cite{bubley1997path}
says that constructing a coupling for single steps of the Markov
chains from certain pairs of configurations in $\Omega$ is sufficient
for establishing an upper bound on the mixing time. To use this
approach, one must represent the state space $\Omega$ as the vertex set of a
connected (finite or infinite) graph $G_{\Omega}$ with a function
$\hat{D}\geq 1$ defined on the edges of $G_{\Omega}$. 
 The \emph{path metric} $D$ corresponding
to $\hat{D}$ is the shortest path distance on the graph with edge
weights given by $\hat D$, i.e.,
\begin{equation}
  \label{eq:pathmetric}
  D(X,Y)  = \inf_{\gamma\colon X\to Y}\set{
  \sum_{i=0}^{\abs{\gamma}-1} \hat{D}(\gamma_{i},\gamma_{i+1}) },
\end{equation}
where the infimum is over nearest-neighbor paths
$\gamma = (\gamma_{0},\gamma_{1}, \dots, \gamma_{\abs{\gamma}})$ in $G_{\Omega}$
with $\gamma_{0}=X$ and
$\gamma_{\abs{\gamma}}=Y$.   The path coupling technique requires
  that $\hat D$ is a \emph{pre-metric} on $G_{\Omega}$, which by
  definition means that 
\begin{equation*}
  \label{eq:premetric}
  \hat D(X,Y) = D(X,Y) \qquad \text{for all edges $\{X,Y \}$ of $G_{\Omega}$}.
\end{equation*}

To establish a rapid mixing regime for the
single-center dynamics we will apply the version of Bubley and Dyer's
path coupling theorem stated below. In the theorem and what follows,
the diameter with respect to $D ( \cdot, \cdot)$,
 of the graph on $\Omega$ is 
 \begin{equation}
 \text{diam}(\Omega) = \sup_{X,Y \in \Omega} D(X,Y). 
 \end{equation}

\begin{theorem}[{\cite[Corollary 14.7]{levin2017markov}}]
  \label{path coupling thm}
  Suppose the state space $\Omega$ of a Markov chain is the vertex set
  of a connected graph $G_{\Omega}$, and suppose that $\hat{D}$ is a pre-metric
   on $G_{\Omega}$. Let $D$ be the corresponding path metric.

  Suppose that for each edge $\{X_0, Y_0\}$ of $G_{\Omega}$ 
  the following holds: if $p^{(0)}$ and $q^{(0)}$ are the distributions
  concentrated on the configurations $X_0$ and $Y_0$ respectively, then
  there exists a coupling $(X_1, Y_1)$ of the distributions $p^{(1)}$
  and $q^{(1)}$ such that
  \begin{equation*}
    \E\cb{D(X_1, Y_1)}\leq  D(X_0, Y_0)e^{-\alpha}, 
  \end{equation*}
  where $\E$ is the expectation with respect to the Markov chain. Then
  \begin{equation*}
    t_{\mathrm{mix}}(\varepsilon) \leq \left \lceil \frac{ \log(\mathrm{diam}(\Omega)) + \log(1/\varepsilon )}{\alpha} \right
    \rceil.
  \end{equation*}
\end{theorem}
\begin{remark}
  \cite[Corollary 14.7]{levin2017markov} concerns Markov chains on
  finite state spaces, but the proof applies essentially verbatim to
  our context.
\end{remark}

\subsection{Single-center dynamics} 
We will use Theorem~\ref{path coupling thm} to prove that the
single-center dynamics introduced in Section~\ref{sec:OTM-I} are rapidly mixing 
at fugacities $\lambda < 2^{1-d}$.

\begin{theorem}\label{mixing}
  Let $\Lambda \subset \R^d$ be compact and measurable,
  $n= |\Lambda|$, $\gamma\in (0,1)$, and let
  $\lambda=(1-\gamma)2^{1-d}$.  The mixing time of the single-sphere
  dynamics with fugacity $\lambda$ on $\Omega_\Lambda^\tau$ satisfies
   \begin{equation*}
    t_{\text{mix}}(\varepsilon) \leq 
    \left \lceil
      \frac{4n(  \log (2^{d+2}n) + \log(1/\varepsilon))}{\gamma}
    \right \rceil .
  \end{equation*}
for all boundary conditions $\tau$.
\end{theorem}

Before discussing the proof of this bound, we derive
Theorem~\ref{thmMCmixing} from it.

\begin{proof}[Proof of Theorem~\ref{thmMCmixing}]
  Let $\Lam$ be a compact measurable subset of $\R^d$ of volume
  $n$.  To show optimal temporal mixing with constants $b,c>0$, it is
  enough to show that with an arbitrary initial distribution $X_0$,
  $\| X_{\lfloor sn \rfloor} - \mu_{\Lam}^\tau \|_{TV} \le \frac{b}{2}
  n e^{-cs}$, and then use the triangle inequality to bound
  $\| X_{\lfloor sn \rfloor} - Y_{\lfloor sn\rfloor} \|_{TV}$.  In
  other words, setting $\eps =\frac{b}{2} n e^{-cs}$, we want to show
  $\tau_{\mathrm{mix}}(\eps)\le \lfloor sn \rfloor$. Taking
  $b= 2^{d+3}$ and $c=\gamma/4$ and applying Theorem~\ref{mixing} proves this. 
\end{proof}

To establish rapid mixing for the single-center dynamics, we follow
the approach of Vigoda for the discrete hard-core model on bounded
degree graphs~\cite{vigoda2001note}.  This approach makes use of an
extended state space $\Omega^{*}\supseteq \Omega$. In our setting, let
$\Omega^{\tau,*}_{\Lam}$ be the collection of all sets of centers
$X \subseteq \Lam_\Int$ such that each point in $\Lambda$ is covered
by at most two balls of radius $r$ with a center in $X$, i.e.
\begin{equation}
  \label{eqOmegaStar}
  X \in \Omega^{\tau,*}_{\Lam} \iff \text{ for all } x \in \Lambda,
  \quad |\{y \in X : \mathrm{dist}(x, y)< r  \} | \leq 2. 
\end{equation}
The purpose of this extended state space will become clear below when
we introduce a pre-metric.  Note that the boundary conditions $\tau$
play no role in the definition of $\Omega^{\tau,*}_{\Lam}$.  Next we
extend our definition of the single-center dynamics to
$\Omega^{\tau,*}_{\Lam}$.  At each step of the chain:
\begin{enumerate}
\item Pick $x \in \Lam$ uniformly at random.
\item With probability $1/ (1+\lambda)$, remove any $y \in X_t$ with
  $\mathrm{dist}(x,y)\le r$. That is, set
  $X_{t+1} = X_t \setminus B_{r}(x)$.
\item With probability $\lambda/ (1+\lambda)$, attempt to add a center
  at $x$.  Let $X' = X_t\cup \{x\}$. If $x \in \Lam_\Int \setminus \tau$ and
  $\mathrm{dist}(x, X_t) \ge 2r$, then set $X_{t+1} = X'$. If not, set
  $X_{t+1}= X_t$.  That is, we add a center at $x$ if it locally satisfies the
  packing constraints and boundary conditions.
\end{enumerate}  
If $X_t \in \Omega^\tau_{\Lam}$ then the chain will remain in
$\Omega_\Lam^\tau$ and the dynamics agree with the definition given in
Section~\ref{secIntro}.  In addition, a Markov chain that starts in
$\Omega^{\tau,*}_\Lam \setminus \Omega^\tau_{\Lam}$ will eventually
reach $\Omega^\tau_{\Lam}$. Since the chain has a unique invariant
measure when considered on the state space $\Omega_{\Lam}^\tau$, this
shows the chain also has the same unique invariant measure on
$\Omega^{\tau,*}_{\Lam}$, and that the mixing time of the chain on
$\Omega^{\tau,*}_\Lam$ is an upper bound for the mixing time of the
chain on $\Omega^\tau_{\Lam}$.

Throughout the remainder of this section, we fix the dimension $d$,
the region $\Lambda \subset \R^d$, and the boundary conditions
$\tau$. For simplicity we write $\Omega= \Omega_\Lam^\tau$ and
$\Omega^* = \Omega^{\tau,*}_\Lam$.

\begin{lemma}
  \label{lem:stat}
  The stationary distribution of the single-center dynamics on
  $\Omega$ is the distribution of the hard sphere model on $\Omega$.
\end{lemma}

\begin{proof}
  Consider two distinct configurations $X,Y\in \Omega$.  
  The transition density from $X$ to $Y$ is proportional to $\mathbf 1_{|X \Delta Y|=1} \cdot \lam^{|Y|}$.
  Suppose without loss of generality that $Y = X \cup \{x \}$. Let
  $\pi$ denote the density of $\mu$, and let $\pi_U(V)$ denote the
  transition density from state $U$ to state $V$. Then
  $\pi(Y)/\pi(X) = \lam$, and $\pi_X(Y) / \pi_{Y}(X) =\lam$, and so
  the chain is reversible with respect to the hard sphere measure on
  $\Omega$.
\end{proof}
Since the single-center dynamics are a Harris recurrent chain, the
previous lemma implies that $\mu$ is the unique invariant measure for
the dynamics on $\Omega$, and that $p^{(t)}\to \mu$ for all initial
distributions $p^{(0)}$, see,
e.g.,~\cite[Section~3.2]{RobertsRosenthal}.

\subsection{Proof of Theorem~\ref{mixing}}
We begin with some preliminary definitions. For 
$X \in \Omega^*$ let 
\begin{equation}
  \Gamma(X)= \left(\Lam \setminus \Lam_\Int \right)\cup \tau \cup
  \left( \bigcup_{x \in X}B_{2r}(x) \right). 
\end{equation}
This is the `blocked volume' of a configuration $X$ where an
additional center cannot be placed.

For $v\in \Lam$ we write the ball $ B_{2r}(v) $ as the disjoint union
of the \emph{occupied} (or \emph{blocked}) set $O_X(v)$ and the
\emph{unoccupied} (or \emph{free}) set $U_X(v)$,
\begin{equation}
  O_X(v)= B_{2r}(v) \cap \Gamma(X), \qquad 
  U_X(v)= B_{2r}(v)\setminus \Gamma(X) .
\end{equation}

We now use these notions to define a graph and pre-metric on $\Omega^{*}$. For
$X, Y \in \Omega^*$, we say that $X$ and $Y$ are adjacent ($X \sim Y$)
if $X$ has exactly one more center than $Y$, and all the centers in
$Y$ are also in $X$ (or vice versa). That is, $X \sim Y$ if and
  only if $| X \Delta Y | =1$.  We will write $G_{\Omega^{*}}$
  for the graph whose edges are pairs of adjacent vertices, and we
define 
a pre-metric $\hat D(\cdot, \cdot)$ on this graph by
\begin{equation}
  \hat{D}(X, X \cup \{v\})= 2^d- c |O_X(v)|, \qquad c=\frac{\lambda2^d}{2+\lambda 2^d}.
\end{equation}
Note that since $c$ is in $[0,1/2]$ for $\lam$ in $[0, 2^{1-d}]$, we
have that $\hat D(X, Y)$ is in $[ 2^{d-1} ,2^d]$ for all edges
$\{X, Y \}$ of $G_{\Omega^{*}}$. Thus for each edge $\{X,Y\}$ we have
$\hat{D}(X,Y)\geq 1$, and $\hat D$ is a premetric since any path from
$X$ to $Y$ other than $\{X,Y\}$ has length at least $2^{d}$ under the
corresponding path metric.

The pre-metric $\hat D$ is the continuous analogue of the pre-metric
introduced by Vigoda in~\cite{vigoda2001note}. Defining the state
space to be $\Omega^*$ rather than $\Omega$ affects the metric
$D$. Consider a simple example with free boundary conditions in which
$\Lambda$ is a ball of radius $3r/2$. Then
$\Omega = \emptyset \cup \bigcup_{x \in \Lambda_\Int} \{\{x\}\}$. For
the state space $\Omega$ the graph of adjacent states is a star graph
with center $\emptyset$, and so for non-empty distinct
$X,Y \in \Omega$,
$D(X,Y)= \hat{D}(X,\emptyset)+ \hat{D}(Y,\emptyset) = 2^{d+1}$. In
contrast, for the state space $\Omega^*$, we have that
$D(X,Y)\leq \hat{D}(X,X \cup Y)+ \hat{D}(Y,X\cup Y) = 2^{d+1}(1- c)$.
This is relevant in our proof when we bound the distance between a
pair of configurations using the triangle inequality applied with a
third configuration that is in $\Omega^* \setminus \Omega$ (see
\eqref{use star}).

To apply Theorem~\ref{path coupling thm} we will couple adjacent
configurations using the following coupling.
\begin{defn}[The identity coupling for the single-center dynamics]
  The identity coupling for the single-center dynamics is defined as
  follows. If $X_t$ and $Y_t$ are separate instances of the single-center
  dynamics for $\mu_\Lam^\tau$ at time $t$, we couple them in a
  Markovian manner via the transition rule
  \begin{itemize}
  \item Choose a point $x\in\Lam$ uniformly at random.
  \item With probability $1/(1+\lam)$, in both $X_t$ and $Y_t$ delete
    any center in $B_r(x)$ to form $X_{t+1}$ and $Y_{t+1}$
    respectively.
  \item With probability $\lam/(1+\lam)$, attempt to add a center at
    $x$ in both $X_t$ and $Y_t$.
  \end{itemize}
\end{defn}

Consider $X,Y \in \Omega^*$ with $Y= X \cup \{v\}$. Let $X'$ and $Y'$
denote the resultant states after one step of the Markov chains
coupled according to the above identity coupling, and let 
\begin{equation}
  \Delta= D(X',Y')- D(X,Y)
\end{equation}
denote the random change in distance between configurations.  The next
lemma bounds the expectation of $\Delta$.
\begin{lemma}
  \label{expected-delta}
  Let $X,Y \in \Omega^*$ such that $Y= X \cup \{v\}$.  Let
  $\lambda= (1-\gamma) 2^{1-d}$, with $\gamma \in (0,1)$. Then
  \begin{equation}
    \label{change}
    \E\cb{\Delta}\leq \frac{2^d(2c-1)}{n(1+\lambda)}=  -  \frac{\gamma
      2^d}{(2-\gamma) (1+\lam) n}  <0.
  \end{equation} 
\end{lemma}

\begin{proof}
  Let $Y=X\cup\{v\}$. The change in distance $\Delta$ is a random variable whose
  value is a function of the current configurations of the chains, the
  random point $w$ chosen in a single step of the coupling, and
  whether or not the coupling tries to add or remove spheres.  We
  begin with a case analysis of $\Delta$.  Throughout the proof we
    will use that $\hat D(X,Y) = D(X,Y)$ for edges $\{X,Y\}$. 
  \begin{enumerate}
  \item Let $A_1$ be the event the center $v$ is removed from $Y$,
    i.e., the chain removes spheres and $w$ lies within distance $r$
    of $v$.  The probability of this event is $1/(n(1+\lam))$.  After
    $A_1$ occurs, $X'=Y'$, and so $\Delta = -D(X,Y)$.  It follows that
    \begin{align}
      \label{eq:Delta1}
      \E\cb{\Delta \cdot 1_{A_{1}}} &= -\frac{1}{n(1+\lam)}D(X,Y)=
                                      -\frac{2^d-c|O_X(v)|}{n(1+\lam)} 
    \end{align}
		
  \item Let $A_2$ denote the event that a center is added to $X$ but
    not $Y$. This occurs when $w$ lies in $U_X(v)$ and the coupling
    attempts to add a sphere, as $U_{X}(v)$ is blocked in $Y$ and not
    blocked in $X$.  In this case we have
    $\Delta = D(X \cup \{w\},Y)- D(X,Y).$ It follows that
    \begin{align}
      \label{eqDelta2}
      \E\cb{\Delta \cdot 1_{A_{2}}} 
      &= \frac{\lam}{n(1+\lam)}\int_{U_{X}(v)} (D(X\cup\{w\},Y)-D(X,Y))\, dw.
    \end{align}
	
  \item Let $A_3$ be the event that a new center $w$ is added to both
    $X$ and $Y$. Note that this event only occurs when
    $w \in \Lam \setminus \Gamma(Y)$ and the coupling adds a center.
    In this case
    \begin{equation*}
      \Delta= -c |\{x \in U_X(v): \text{ $x$ is blocked by the new
        center $w$}\}|.
    \end{equation*}
    For $x \in U_X(v)$, let $A_3^x$ be the event that $x$ becomes
    blocked by the new center, i.e., that $X'=X \cup \{w\}$,
    $Y'=Y\cup\{w\}$ and $x \in O_{X \cup \{w\}}(v)$. In order for the
    event $A_3^x$ to occur, it must be the case that
    $w \in B_{2r}(x) \setminus \Gamma(Y)$. Hence
    \begin{align}
      \nonumber
      \E\cb{\Delta \cdot 1_{A_{3}}} 
      &=  \E\cb{\int_{U_x(v)} - c 1_{A_3^x}\, dx} \\
      \nonumber
      &= -\frac{c\lam}{n(1+\lam)} \int_{U_x(v)} \int_{\Lam} \, 1_{w \in
        B_{2r}(x) \setminus \Gamma(Y)} \,dw\, dx \\ 
      \label{eq:Delta3}
      &=  -\frac{c\lam}{n(1+\lam)}\int_{U_{X}(v)}\abs{B_{2r}(x)\setminus\Gamma(Y)}\,dx
    \end{align}
		
  \item Let $A_4$ be the event that at least one center is removed in
    both $X$ and $Y$, and $v$ is not removed. Let $S_w$ be the set of
    centers removed; since $w \not \in B_r(v)$ we have
    $S_w= X \cap B_r(w)= Y \cap B_r(w)$.  In this case,
    \begin{equation*}
      \Delta= c |\{x \in O_X(v): \text{ $x$ is no longer blocked after
        $S_w$  is removed}\}|. 
    \end{equation*}
    For $x \in O_X(v)$, let $A_4^x$ be the event that
    $X'= X \setminus S_w$, $Y'= Y \setminus S_w$, and
    $x \in U_{X\setminus S_w}(v)$.  If $A_4^x$ occurs there is a
    center $b_x\in X$ that is the closest center to $x$ that blocks
    $x$. In particular, $b_x\in S_w$, and hence $w\in B_r(b_x)$. Using
    this observation we obtain
    \begin{align}
      \nonumber
      \E\cb{\Delta \cdot 1_{A_{4}}} &= \E \cb{
                                      \int_{O_X(v)} c 1_{A_4^x}\, dx} \\
      \nonumber
                                    &\leq \frac{c}{n(1+\lam)} \int_{\Lam}\int_{O_X(v)} \, 1_{w
                                      \in B_r(b_x) }  \, dx \,dw\\
      \label{eq:Delta4}
                                    &=\frac{c\abs{O_{X}(v)}}{n(1+\lam)}.
    \end{align}
  \end{enumerate}
	
  The events $A_1, A_2, A_3,$ and $A_4$ are mutually exclusive and
  exhaustive, so
  \begin{equation}
    \label{delta sum}
    \E[\Delta]=\E\cb{\Delta \cdot
      \sum_{i=1}^{4}1_{A_{i}}} . 
  \end{equation}
  To derive an upper bound on $\E[\Delta]$ we first need to estimate
  the integrand in~\eqref{eqDelta2}.  We will use the triangle
  inequality with the configurations $Y\cup \{w\}$, $X \cup \{w\}$,
  and $Y$.  Temporarily deferring the justification of the use of the
  triangle inequality, note that since $c\geq 0$,
  $D(Y \cup \{w\},X\cup \{w\}) \leq D(Y, X)$. Further, by definition,
  $D(Y \cup \{w\}, Y)= 2^d-c |B_{2r}(w) \cap \Gamma(Y)| $. Hence by
  the triangle inequality
\begin{align} 
  \label{use star} 
  D(X \cup \{w\}, Y)- D(X,Y) 
  &\leq D(Y \cup \{w\},X\cup \{w\}) \nonumber  
       +D(Y \cup \{w\}, Y)-D(X,Y)  
      \nonumber \\
    &\leq 2^d-c |B_{2r}(w) \cap \Gamma(Y)|.
\end{align}
To justify this use of the triangle inequality we must establish that
$X \cup \{v, { w} 
\} \in \Omega^*$.  Note that no point of $\Lam$ is
covered by three balls of radius $r$ whose centers are in $Y$ because
$Y\in \Omega^*$. No point that is covered by $B_{r}({ w})$  
is covered by
$B_r(u)$ for some $u \in X$ since $ w$ 
is added to $X$ by the Markov
chain. It follows that no point in $\Lam$ is covered three times by
$Y\cup \{  w\}$,  i.e., $Y\cup  \{ w\} 
\in \Omega^*$.

Inserting the estimates given in \eqref{eq:Delta1}--\eqref{eq:Delta4}
into \eqref{delta sum} we obtain
\begin{align*}
  \E\cb{\Delta}
  &\leq \frac{1}{n(1+\lambda)}\bigg( -(2^{d}-c|O_{X}(v)|)
     + \lam\int_{U_{X}(v)}(2^{d}-c|B_{2r}(w)\cap \Gamma(Y)|)\,dw \\
               &  \qquad\qquad\qquad\qquad - c \lambda \int_{U_X(v)} 
                 |B_{2r}(x) \setminus\Gamma(Y)| \,dx 
                 + c|O_X(v)|\bigg)\\
               &=\frac{1}{n(1+\lambda)} \left(-2^d +2c |O_X(v)|
                 +\lambda 2^d (1-c)|U_X(v)| \right),
\end{align*}
where the last line follows from $\abs{B_{2r}(x)\cap\Gamma(Y)}+\abs{B_{2r}(x)\setminus\Gamma(Y)}=2^{d}$.
Since $|U_X(v)|+|O_X(v)|=2^d$ and $2c=\lambda 2^d (1-c)$, it follows
that
\begin{equation*}
  \E\cb{\Delta}\leq \frac{2^d(2c-1)}{n(1+\lambda)}=  -  \frac{\gamma 2^d}{(2-\gamma) (1+\lam) n} \, .\qedhere
\end{equation*} 
\end{proof}

Now we deduce Theorem~\ref{mixing} from Theorem~\ref{path coupling thm}.
\begin{proof}[Proof of Theorem~\ref{mixing}]
  First we bound the diameter of $\Omega^*$ with respect to $D(\cdot, \cdot)$. Note that if
  $X \in \Omega^*$ then $|X|\leq 2n$ since each ball covers one unit
  of volume and each point cannot be covered more than
  twice. Recall that we
  defined the graph $G_{\Omega^{*}}$
by putting an edge 
    between
     configurations that differ by exactly one
    sphere location. It
  follows that the combinatorial diameter of $G_{\Omega^{*}}$ 
  is bounded above by $4n$. For two adjacent states $X$ and $Y$,
  $D(X, Y)=\hat{D}(X,Y)\le 2^d$, and hence
  $\text{diam}(\Omega^{*})\leq n2^{d+2}$.  Note also that
    $G_{\Omega^{*}}$ is connected since there is a path from any
    configuration to the empty configuration.

  Next we find a suitable value for $\alpha$ in the statement of
  Theorem~\ref{path coupling thm}.  Let $X_0= X $ and $Y_0= X \cup
  \{v\}$. Applying Lemma~\ref{expected-delta} we obtain 
  \begin{align*}
\E \cb{D(X_1, Y_1)} 
     & =D(X_{0}, Y_{0}) \ob{1 + \frac{\E \cb{\Delta(X_0, Y_0)}}{D(X_{0},Y_{0})}} \\
     &\le D(X_{0}, Y_{0}) \ob{1 - \frac{\gamma}{n (2-\gamma)(1+\lam) }}  \\
     &\le D(X_{0}, Y_{0}) \exp \left [  - \frac{\gamma}{n (2-\gamma)(1+\lam) } \right ]  \\
     &\le D(X_{0}, Y_{0}) e^{  - \frac{\gamma}{4 n }} \, .
\end{align*}
The first inequality used that $\E[\Delta]<0$ and
$D(X_0,Y_0) \le 2^d$, and the last used that $1+\lam \le 2$.
Since $\hat D$ is a pre-metric, we have verified the hypotheses of
Theorem~\ref{path coupling thm} with $\alpha= \gamma/ 4n$, which gives
Theorem~\ref{mixing}.
\end{proof}

\section{Spatial and temporal mixing}
\label{secSpatial}

In this section we prove Theorems~\ref{thmSingleTempToSpatial}
and~\ref{thm:heatbathConverse} following the approach of Dyer,
Sinclair, Vigoda, and Weitz~\cite{dyer2004mixing} who proved similar
results for the discrete hard-core model on the integer lattice
$\Z^d$.  At the heart of this technique is the idea of disagreement
percolation, bounding the distance that a disagreement between two
copies of a Markov chain can typically spread in a fixed number of
steps.  This idea appeared in~\cite{stroock1992logarithmic} in the
context of spatial and temporal mixing with further refinements and
applications due to van den Berg~\cite{van1993uniqueness}.

The main complication in the continuous setting is that there is a
  richer variety of domains and boundary conditions to consider. To
  handle this we will need the following lemma about the volume of
parallel sets in Euclidean space.
\begin{lemma}[Fradelizi--Marsiglietti~\cite{fradelizi2014analogue}]
  \label{lemParallelVolume}
  Suppose $L \ge r$, then
  \begin{equation*}
    |\Lam^{(L)} | \le  \frac{L^d}{r^d} |\Lam^{(r)}| \,.
  \end{equation*}
  In particular, for $L \ge r$ we have
  $|\Lam_{\mathrm{Int}}^{(L)}| \le \frac{L^d}{r^d} | \Lam|$.
\end{lemma}
\begin{proof}
  For $\Lam, B \subset \R^d$ compact with $B$ convex, Fradelizi and
  Marsiglietti~\cite[Proposition 2.1]{fradelizi2014analogue} proved
  that the function $|s\Lam + B| -s^d |\Lam|$ is non-decreasing
  and continuous as a
  function of $s$ on $\R_+$, where $s\Lam+B$ is the Minkowski sum of
  $s\Lam$ and $B$. 
  In particular,
  \begin{equation*}
    \abs{\Lam + B_{L}(0)} - \abs{\Lam} \leq \abs{\frac{L}{r}\Lam +
      B_{L}(0)} - \left(\frac{L}{r}\right)^{d}\abs{\Lam},
  \end{equation*}
  where we have used the continuity in $L$ to obtain the result for the
    open ball $B_{L}(0)$ instead of its closure.  Since $L\ge r$ this
  implies
  \begin{align*}
 \abs{\Lam + B_{L}(0)} \leq \abs{\frac{L}{r}\Lam + B_{L}(0)}  = \left(
    \frac{L}{r}\right)^d \abs{\Lam + B_{r}(0)} \,, 
  \end{align*}
  and the first claim follows since $\Lam^{(L)} = \Lam + B_{L}(0)$ and
  $\Lam^{(r)} = \Lam + B_{r}(0)$. The second claim follows since
  $\Lam = \Lam_{\mathrm{Int}}^r$.
\end{proof}

\subsection{From temporal to spatial mixing}
\label{sec:preparatory-lemmas}
The following lemma bounds how fast a disagreement between two
copies of the single-center dynamics can spread. This is a continuum
variant of~\cite[Lemma 3.1]{dyer2004mixing}.

For $\Lam'\subset \Lam \subset \R^d$ we write $X_t\cb{\Lam'}$ to
denote the projection of $X_{t} \subset \Lam_{\mathrm{Int}}$ to the
set $\Lam'_{\mathrm{Int}}$. In words, $X_{t}[\Lam']$ is the set
  of centers of spheres in $X_{t}$ that are entirely contained in
  $\Lam'$. 
\begin{lemma} 
  \label{lem:infPropagationSingle}
  Let $X_{t}$ and $Y_{t}$ be two copies of the single-center dynamics
  for the hard sphere model on $\Lam$ with boundary conditions
  $\tau_X$ and $\tau_Y$ and initial conditions
  $X_0,Y_0 \in \Omega_{\Lam}^{\tau,*}$. Suppose both
  $X_{0}\triangle Y_{0}$ and $\tau_{X}\triangle\tau_{Y}$ are contained in
  $A\subset\Lam$. Let $B \subset \Lam$ with
  $ \mathrm{dist}(A,B_{\mathrm{Int}}) = s >0$. Then for all positive
  $\eta \le \frac{ s}{e^{2}r \cdot 4^{d+1}}$, under the identity
  coupling we have
  \begin{equation}
    \Pr\cb{ X_{\lfloor \eta n \rfloor}[B] \ne Y_{ \lfloor \eta n \rfloor} [B]}  \le
     |B| \cdot  e^{-s/(4r)},
  \end{equation}
  where $n = | \Lam|$.   
  
\end{lemma}
\begin{proof}[Proof of Lemma \ref{lem:infPropagationSingle}]
  We couple $X_{t}$ and $Y_{t}$ via the identity coupling. Say
  $t^{\prime}$ is the smallest $t$ so that $X_{t} [B] \ne Y_{t} [B]$.
  That is, there is a center in $B_{\mathrm{Int}}$ in one
  configuration but not the other.  Since removing a center will not
  create a disagreement, at step $t^{\prime} $ exactly one of the
  Markov chains must add a center at some $w\in B_{\mathrm{Int}}$.  In
  order for the update point $w$ to be added to only one of the Markov
  chains, it must be that $B_{2r}(w)$ contains a point $y$ of
  disagreement, meaning that $y$ is a center in one of the configurations but not the other. 

  Proceeding further, if
  $X_{\lfloor \eta n \rfloor}[B] \ne Y_{\lfloor \eta n \rfloor}[B]$
  then there must be a connected (overlapping) chain of balls of
  radius $2r$ joining $A$ to $B_{\mathrm{Int}}$ with the property that
  there is a point of disagreement in each update ball.  In
  particular, the balls must be ordered in time to propagate a
  disagreement forward.  We call such a chain of balls an
  \emph{ordered chain}.  With $s = \mathrm{dist}(A,B_{\mathrm{Int}})$,
  there must be an ordered chain of at least
  $m = \lceil \frac{s}{4r} \rceil$ balls connecting $A$ to
  $B_{\mathrm{Int}}$.

  In any ordered chain each ball must intersect the last ball added to
  the chain, so the probability of extending a chain of balls of
  radius $2r$ by one ball is at most $4^d/n$. The probability of
  forming an ordered chain of $\ell$ balls of radius $2r$ with the
  final ball centered in $B_{\mathrm{Int}}$ is thus at most
  \begin{equation}
    \binom{\lfloor \eta  n \rfloor} { \ell} \frac{\abs{B_{\mathrm{Int}}}}{n} \ob{
  \frac{4^d}{n}}^{\ell-1},
  \end{equation}
  where we have neglected the constraint that a disagreement must be
  created in each update ball.  This upper bounds the probability of a
  disagreement in $B$ at time $\lfloor \eta n \rfloor$ by
  \begin{equation*}
    \frac{\abs{B_{\mathrm{Int}}}}{4^d}
    \sum_{\ell =  m }^{\lfloor \eta n \rfloor} \binom {\lfloor \eta n \rfloor} { \ell}
    \ob{\frac{4^d}{n}}^{\ell} 
    \leq
    \frac{\abs{B_{\mathrm{Int}}}}{4^d}
    \sum_{\ell =  m }^{\infty}
    \ob{\frac{e \eta4^d}{\ell}}^{\ell}   
    \leq \frac{\abs{B_{\mathrm{Int}}}}{4^d} \sum_{\ell =  m }^{\infty} \ob{\frac{s }{4re  \ell}}^{\ell}  \,.
  \end{equation*}
  The first inequality used ${M\choose \ell} \leq (eM/\ell)^{\ell}$ 
  and the second used the hypothesis on $\eta$, i.e., $\eta < \frac{
    s}{e^{2}r  \cdot 4^{d+1}}$. The ratio of consecutive
    terms in the summation is at most $1/e$, so
the entire series is bounded by twice the first term. This gives an upper bound of 
\begin{equation*}
  \frac{2\abs{B_{\mathrm{Int}}}}{4^d} \ob{\frac{s }{4er m}}^{m} \leq
  \frac{2\abs{B_{\mathrm{Int}}}}{4^d}e^{-s/(4r)}  \le |B|
  e^{-s/(4r)}  \,,
\end{equation*}
where the first inequality is due to the definition of $m$ as a ceiling.
\end{proof}
\begin{remark}
  \label{rem:BC}
  An inspection of the preceding proof reveals that it also applies to
  the situation in which the boundary conditions
  $\tau_{X}=\tau_{X}(t)$ and $\tau_{Y}=\tau_{Y}(t)$ change in time,
  provided $\tau_{X}(t)\triangle\tau_{Y}(t)\subset A$ for all $t$. In
  this situation the configurations are in
  $\Omega^{\tau,*}_{\Lam}$ as they may not satisfy the boundary
  conditions. 
\end{remark}

The next lemma shows that optimal temporal mixing implies what is
called \emph{projected optimal mixing}~\cite[Lemma
4.1]{dyer2004mixing}. Recall the definition of $\norm{\cdot}_{\Lam'}$
from Section~\ref{sec:SSM-I}.
\begin{lemma} 
  \label{lem:Proj-MixingSingle}
  If the single-center dynamics has optimal temporal mixing on $\R^d$
   with constants $b, c>0$ then there
  exist constants $b^{\prime},c^{\prime}>0$ such that, for any
  compact measurable $\Lam \subset \R^d$, any boundary condition $\tau$,
  any two instances $X_{t}$ and $Y_{t}$ of the dynamics on
  $\Omega_\Lam^\tau$, and any measurable $\Lam' \subset \Lam$,
  we have that
  \begin{equation*}
    \norm{X_{\lfloor \eta n\rfloor} - Y_{\lfloor \eta n
        \rfloor}}_{\Lam'}  
    \leq b^{\prime}\abs{\Lam'} e^{-c^{\prime}\eta}
  \end{equation*}
  for any $\eta > e^{-2} 4^{-d-1}$, where $n = \abs{\Lam}$.  The same
  conclusion also holds if
  $X_0, Y_0 \in \Omega_{\Lam}^{\tau,*}$, as long as
  $X_0[A_R],Y_0[A_R] \in \Omega_{A_R}$, with $A_R$ as defined below in
  the proof.
\end{lemma}
\begin{proof} 
Fix $\eta > e^{-2} 4^{-d-1}$ and let $R = \eta e^{2} r \cdot 4^{d+1} $.
Define  $A_{R}$  to be
\begin{equation*}
A_{R} =  \{x\in \Lam :  \mathrm{dist}(x,\Lam'_{\mathrm{Int}})\leq R\}.  
\end{equation*}
If $\abs{\Lam'}<1$, then the total variation distance is zero since
both $X_{t}[\Lam']$ and $Y_{t}[\Lam']$ are the empty set since no
spheres fit inside $\Lam'$. Hence we may assume that $|\Lam'|\geq 1$,
and our assumption on $\eta$ implies that
$|A_R| \ge |A_r| = |\Lam'| \geq 1$.
 
The proof proceeds by defining auxiliary Markov chains $X_{t}^{R}$ and $Y_{t}^{R}$ on
$A_{R}$ that imitate $X_{t}$ and $Y_{t}$ closely, and then
using the triangle inequality:
\begin{equation}
  \label{eq:Proj-TriangleSingle}
  \norm{X_{\lfloor \eta n \rfloor} - Y_{\lfloor \eta n \rfloor}}_{\Lam'} \leq
  \norm{X_{\lfloor \eta n \rfloor}-X^{R}_{\lfloor \eta n \rfloor}}_{\Lam'} +
  \norm{X_{\lfloor \eta n\rfloor}^{R}-Y_{\lfloor \eta n \rfloor}^{R}}_{\Lam'} + \norm{Y_{\lfloor \eta n \rfloor}^{R} - Y_{\lfloor \eta n \rfloor}}_{\Lam'}.
\end{equation}

The definition of $A_{R}$ will ensure that it is unlikely that
information can pass from outside $A_{R}$  to $\Lam'$, which will
ensure the first and third terms are small. The second term will be
handled by the optimal temporal mixing hypothesis.

In detail, we define the Markov chains $X_{t}^{R}$ and $Y_{t}^{R}$ to be  empty
 outside of $(A_R)_{\mathrm{Int}}$ 
for all $t$, and to agree with $X_{0}$ and $Y_{0}$ respectively
inside $(A_R)_{\mathrm{Int}}$ at $t=0$. The two chains have the same dynamics:
  \begin{itemize}
\item  Uniformly select an update point $x$ from $\Lam$.
\item If $x \notin A_R$, do nothing.
\item Otherwise perform an update of the chain, with the configuration
  outside $A_{R}$ held empty as a boundary condition. Formally, the
  boundary condition for this update is
  $\tau_R = \Lam \setminus (A_R)_{\mathrm{Int}}$.
\end{itemize}
$X_{t}^{R}$ and $Y_{t}^{R}$ are lazy dynamics on $A_{R}$: with
probability $1-|A_{R}|/n$ nothing occurs, otherwise an update on
$A_{R}$ is performed.

We couple $X_{t}^{R}$ with $X_{t}$ by a variant of the identity
coupling: if $X_{t}$ updates at a point outside $A_{R}$ then
$X_{t}^{R}$ does nothing, otherwise attempt the same update. The
projections of $X_{t}$ and $X_{t}^{R}$ to $\Lam'$ are both copies of
the hard sphere model on $\Lam'$ with boundary conditions and initial
conditions that only differ outside $(A_R)_{\mathrm{Int}}$, see
  the third bullet above. As a result, 
Lemma~\ref{lem:infPropagationSingle} and Remark~\ref{rem:BC} imply that
\begin{equation*}
  \norm{X_{\lfloor \eta n\rfloor }-X^{R}_{\lfloor \eta n \rfloor}}_{\Lam'} 
      \le |\Lam'| e^{-R/(4r)}
      = |\Lam'|e^{-\eta e^2 4^d}.
\end{equation*}
The application of Lemma~\ref{lem:infPropagationSingle} is valid by the
definition of $R$ and $A_{R}$, i.e., that
$\mathrm{dist}(\Lam'_{\mathrm{Int}},A_{R}^{c}) > R$. In particular,
this holds even if $X_0 \in \Omega_{\Lam}^{\tau,*}$.  Exactly the same
reasoning and bound apply to
$\norm{Y_{\lfloor \eta n\rfloor }^{R} - Y_{\lfloor \eta
    n\rfloor}}_{\Lam'}$.

As the configurations of $X_{t}^{R}$ and $Y_{t}^{R}$ agree outside
$A_{R}$ and $X_0[A_R],Y_0[A_R] \in \Omega_{A_R}$, the hypothesis of
optimal temporal mixing applies. Hence the second term of
\eqref{eq:Proj-TriangleSingle} is small provided the chain takes
enough steps. There is probability $|A_{R}|/n$ that the update point
lies in $A_{R}$.  So in $\eta n$ steps, we expect $\eta |A_{R}|$
updates to occur in $A_{R}$.  By a Chernoff bound at least
$\eta |A_R|/2$ updates occur in $A_{R}$ with probability at least
$1-e^{-\eta |A_R|/8} \ge 1- e^{-\eta/8}$ since $|A_R| \ge 1$.  This
gives
\begin{equation*}
  \norm{X_{\lfloor \eta n \rfloor}^{R}-Y_{\lfloor \eta n \rfloor}^{R}}_{\Lam'} \leq
  \norm{X_{\lfloor \eta n \rfloor}^{R}-Y_{\lfloor \eta n \rfloor}^{R}}_{A_{R}}
  \leq b\abs{A_{R}}e^{-c\eta/2} + 2e^{-\eta /8}.
\end{equation*}
The first inequality has used that the total variation distance is
weakly decreasing when projecting to subsets. For the second we have
applied the definition of optimal temporal mixing and used a union
bound to ensure both $X_{\lfloor \eta n \rfloor}^{R}$ and
$Y_{\lfloor \eta n \rfloor}^{R}$ have taken $\eta |A_R|/2$ steps.
Putting these bounds together with~\eqref{eq:Proj-TriangleSingle}, we
have
\begin{align*}
  \norm{X_{\lfloor \eta n \rfloor} - Y_{\lfloor \eta n \rfloor}}_{\Lam'} &\le 2 |\Lam'|e^{-\eta e^2 4^d}  + b\abs{A_{R}}e^{-c\eta/2} + 2e^{-\eta /8}  \\
  &\leq 2 |\Lam'|e^{-\eta e^2 4^d}  + b(R/r)^d e^{-c \eta/4}  |\Lam'| e^{-c\eta/4} + 2e^{-\eta /8}
\end{align*}
since Lemma~\ref{lemParallelVolume} implies $ |A_{R}| \le (R/r)^d |\Lam'| $.
 With
 \begin{equation*}
 b' =\sup_{\eta \ge e^{-2}4^{-d-1}}b (\eta e^2 4^{d+1})^{d}e^{-c\eta/4} +4 \quad \text{ and } \quad c' = \min \{ c/4,   1/8\},
 \end{equation*}
this proves the claim. 
\end{proof}

Using these lemmas we prove Theorem~\ref{thmSingleTempToSpatial}. Our
proof follows that of~\cite[Theorem~2.3]{dyer2004mixing}.

\begin{proof}[Proof of Theorem~\ref{thmSingleTempToSpatial}]  
  Fix $\lam$ and $d$, and suppose that the single-center dynamics for
  the hard sphere model on $\R^d$ at fugacity $\lam$ exhibits optimal
  temporal mixing with constants $b, c$.  Let $\Lam \subset \R^d$ be
  compact and measurable, and suppose $\tau, \tau'$ are two boundary
  conditions on $\Lam$.  Let $\Lam' \subset \Lam$ be measurable, and
  let $s =\mathrm{dist} (\tau \triangle \tau', \Lam'_{\mathrm{Int}} )$.
 
  Let $Z_t$ be a copy of the single-center dynamics with stationary
  distribution $\mu_{\Lam}^\tau$ and let $Z'_t$ be a copy of the
  dynamics with stationary distribution $\mu_{\Lam}^{\tau'}$, and take
  both initial conditions to be the same sample from
  $\mu_{\Lam}^{\tau'}$. In particular $Z'_t$ is distributed as
  $\mu_{\Lam}^{\tau'}$ for all $t$ (and thus
  $Z_t' \in \Omega_{\Lam}^{\tau'}$). On the other hand, we only
    know $Z_0 \in \Omega_{\Lam}^{\tau,*}$ since the initial condition
  might violate the boundary condition
  $\tau$.  
 
  We have, by the triangle inequality,
  \begin{equation*}
    \| \mu_{\Lam}^\tau - \mu_{\Lam}^{\tau'} \|_{\Lam'} =   \| \mu_{\Lam}^\tau - Z'_t \|_{\Lam'}   \le  \|  \mu_{\Lam}^\tau - Z_t  \|_{\Lam'} +   \| Z_t - Z'_t \|_{\Lam'} \,,
  \end{equation*}
  for any choice of $t$.  From Lemma~\ref{lem:Proj-MixingSingle}, we
  have projected optimal mixing, and so if we take
  \begin{equation*}
    t =  \left \lfloor  \frac{ s n}{e^{2} r \cdot 4^{d+1} } \right \rfloor \,,
  \end{equation*}
  we have
  \begin{equation*}
    \|  \mu_{\Lam}^\tau - Z_t  \|_{\Lam'}  \le b' |\Lam'| e^{- c' s/ (e^{2} r 4^{d+1}  )}  \,.
  \end{equation*}
  We can apply Lemma~\ref{lem:Proj-MixingSingle} even though
  $Z_0 \in \Omega_{\Lam}^{\tau,*}$ since $Z_0[A_R]\in\Omega_{A_R}$
  with $A_R$ as defined in the proof of
  Lemma~\ref{lem:Proj-MixingSingle}.

  The centers of $Z_t$ and $Z_t'$ outside of $\Lam'$ determine the
  boundary conditions of the projected chain restricted to
  $\Lam'$. The symmetric difference of these boundary conditions are
  contained in $(\tau \triangle \tau')^{(2r)}$.  Therefore, by
  Remark~\ref{rem:BC} our choice of $t$ allows us to apply
  Lemma~\ref{lem:infPropagationSingle}, which gives
  \begin{equation*}
    \| Z_t - Z'_t \|_{\Lam'} \le |\Lam'| e^{-(s-2r)/(4r)} \,.
  \end{equation*}

  Setting $\beta=b'+e^{1/2}$ and $\alpha = \min \{   c'/ (e^{2} r 4^{d+1}  )
  ,1/(4r) \}$ and putting these bounds together gives
  \begin{equation*}
    \| \mu_{\Lam}^\tau - \mu_{\Lam}^{\tau'} \|_{\Lam'} \le  b' |\Lam'|
    e^{- c' s/ (e^{2}r 4^{d+1}  )}  +  |\Lam'| e^{-s/(4r)}   \le
    \beta |\Lam'| e^{- \alpha s}\, .\qedhere
  \end{equation*}
\end{proof}

\subsection{From spatial to temporal mixing}
Here we will show that strong spatial mixing implies that the
heat-bath dynamics with radius $L \ge L_0(d,\alpha, \beta)$ exhibits
optimal temporal mixing (Theorem~\ref{thm:heatbathConverse}).  Along with
Theorem~\ref{thmHeatBathTemptoSpatial}, this shows that strong spatial
mixing and optimal temporal mixing of the heat bath dynamics are
essentially equivalent.

\begin{proof}[Proof of Theorem~\ref{thm:heatbathConverse}]
  Assume the hard-sphere model on $\R^d$ exhibits strong
  spatial mixing with constants $\alpha$ and $\beta$. We will prove
  optimal temporal mixing for the heat-bath dynamics with update
  radius $L = Kr$, for $K$ to be chosen large enough in the course of
  the proof.

  We construct a path coupling using Hamming distance.  That is,
  $D(X,Y) = | X \triangle Y|$, the number of centers in the symmetric
  difference of $X$ and $Y$. If $|\Lam|=n$, then at most $n$ centers
  can fit in a valid configuration, and so the diameter of
  $\Omega_\Lam^\tau$ under Hamming distance is at most $2 n$.

  Suppose $X_t$ and $Y_t$ are two copies of the radius-$L$ heat-bath
  chain on $\Omega_\Lam^\tau$, with $X_0 = Y_0 \cup \{u \}$.  Again we
  use an identity coupling to couple the chains: we choose the same
  update ball in each chain; if the boundary conditions agree, we make
  the same update.  If the boundary conditions disagree, then we will
  choose a specific coupling detailed below.
  
  Let $\Delta = D(X_1, Y_1) -D(X_0,Y_0)$ under this coupling.  If $x$
  is the random center of the update ball and $u \in B_L(x)$, then the
  boundary conditions agree and so with probability $1$, $X_1 = Y_1$,
  and so $\Delta= -1$.  This occurs with probability
  \begin{align*}
    \Pr[ u \in B_L(x)] &= \frac{|B_L(u)| }{| \Lam_{\mathrm{Int}}^{(L)}| } = \frac{K^d}{N }  \,,
  \end{align*}
where we set $N = | \Lam_{\mathrm{Int}}^{(L)}|$. 

If $u \notin B_{L+2r}(x)$, then again the boundary conditions of the update ball agree, and so with probability $1$ we will have $X_1 = Y_1 \cup \{u\}$ and so $\Delta=0$. 
 
Finally, if $u \in B_{L+2r}(x) \setminus B_{L}(x)$, the boundary
conditions of the update ball differ by the presence of $u$, and so
the Hamming distance may increase.  We bound the probability that $u$
is in this width $2r$ boundary of the update ball:
 \begin{align*}
   \Pr [ u \in B_{L+2r}(x) \setminus B_L(x)] &=  \frac{  | (B_{L+2r}(u) \cap \Lam^{(L)}) \setminus B_L(u)      |  }{ | \Lam^{(L)}_{\mathrm{Int}}|}     \\
   &\le \frac{| B_{L+2r}(u)
     \setminus B_L(u) | } { | \Lam^{(L)}_{\mathrm{Int}}| } \\
     &= \frac{ (K+2)^d - K^d }{ N }  \, .
\end{align*}
Next we bound the expected increase in Hamming distance in this case
under a specific coupling.

Fix $x \in \Lam_{\mathrm{Int}}^{(L)}$ so that
$u \in B_{L+2r}(x) \setminus B_{L}(x)$.  Let $\tau_X$ be the
boundary condition on $B_{L}(x)$ induced by $X_0$ and let $\tau_Y$ be
the boundary condition induced by $Y_0$.  In particular,
$\tau_X \triangle \tau_Y \subseteq B_{2r}(u)$.  Set 
$t=r(K/8d)^{1/d} -2r$, and let
$A = \{ y \in B_{L}(x) : \mathrm{dist}(y, u) \le t \}$ and
$\overline A = B_{L}(x) \setminus A$. We will bound the change in
  Hamming distance by considering $A$ and $\overline A$ separately;
  the coupling will be chosen to control the change on $\overline A$.

The increase in Hamming distance can be written as the sum of the
increase in Hamming distance restricted to spheres that intersect $A$
plus the increase in Hamming distance restricted to the configuration
in $\overline A$.  An upper bound on the increase in Hamming distance
for spheres intersecting $A$ is twice the maximum number of centers
possible in a valid configuration, which we can upper bound by
$2V_{2r+t }$.

We now turn to $\overline A$.
We can bound the total variation distance between
$\mu_{B_{L}(x)}^{\tau_X}$ and $\mu_{B_{L}(x)}^{\tau_Y}$ restricted
to $\overline A$ using the strong spatial mixing assumption:
\begin{align*}
 \| \mu_{B_{Kr}(x)}^{\tau_X} - \mu_{B_{Kr}(x)}^{\tau_Y}  \| _{\overline A} &\le \beta | \overline{A}^{(r)}|  e^{-\alpha \mathrm{dist} ( \tau_X \triangle \tau_Y,  \overline A ) }\\
 &\le \beta |\overline{A}^{(r)}|  e^{-\alpha \mathrm{dist} ( B_{2r}(u),  \overline A )} \\
 &\le \beta (K+1)^d e^{- \alpha (t-2r) }   \,.
 \end{align*}
 Therefore, there exists a coupling of $X_1, Y_1$ so that $X_1$ and
 $Y_1$ disagree within $\overline A$ with probability at most
 $\beta (K+1)^d e^{- \alpha (t-2r) }$. An upper bound on the increase
 in Hamming distance restricted to $\overline{A}$ is twice
 the maximum number of centers that can be placed in $\overline{A}^{(r)}$,
 which is $2(K+1)^d$.
 Under this coupling we can therefore bound the expected change in Hamming
 distance by
  \begin{equation*}
    \E \left [ \Delta \big| u \in B_{(K+2)r}(x) \setminus B_{Kr}(x) \right ] \le 2V_{2r+t} +   2\beta (K+1)^{2d}   e^{- \alpha (t-2r) }   \,.
\end{equation*}

 Putting this together yields that the expected change in Hamming distance is at most
 \begin{equation*}
\E \left [ \Delta \right ] \le -\frac{K^d }{N } + \frac{ (K+2)^d - K^d  }{  N }   \left( 2V_{2r+t} +   2\beta (K+1)^{2d} e^{- \alpha (t-2r) } \right)   \,.
\end{equation*}
 Now since  $2V_{t+r} = K/4d $ and $(K+2)^d -K^d \le 2d(K+2)^{d-1}$, we have
   \begin{equation*}
\E \left [ \Delta \right ] \le - \frac{1}{N} \left [ K^d -  2d(K+2)^{d-1} \left( \frac{K}{4d} +   2\beta (K+1)^{2d} e^{- \alpha (r (K/8d)^{1/d}-3r) } \right)   \right]  \,,
\end{equation*}
  and choosing $K$ large enough as a function of $d, \alpha, \beta$
  we can ensure that
 \begin{equation*}
\E[ \Delta] \le -\frac{K^{d}}{3} \frac{1}{N} \le -\frac{  1 } {3 n   }
\end{equation*}
since $N = | \Lam_{\mathrm{Int}}^{(L)}| \le K^d n $ by
Lemma~\ref{lemParallelVolume}.  Then combining this bound and the
diameter bound, Theorem~\ref{path coupling thm} gives optimal temporal
mixing.
\end{proof}

\section{Bounds on the critical density}
\label{secDensity}

In this section we prove Theorem~\ref{thmDensityBounds}; this requires
two preparatory results. Recall that 
\begin{equation*}
 \rho(\lam) = \liminf_{n \to \infty}  \frac{1}{n} \E_{Q_n,
   \lam} |\mathbf X | \,,
\end{equation*}
where $Q_n$ is the box of volume $n$ centered at the origin in $\R^d$.
We first give an easy lower bound on $\rho(\lam)$.  This is closely related to an inequality of Lieb~\cite{lieb1963new} that applies to the canonical ensemble. 
\begin{lemma}
  \label{lemEasybound}
For all $d$ and all $\lam \geq 0$,
  \begin{equation*}
    \rho(\lam) \ge \frac{\lam}{1+ 2^d \lam}  \,.
  \end{equation*}
\end{lemma}
\begin{proof}
  Let $\rho_\Lam(\lam) = \frac{1}{|\Lam|} \E_{\Lam, \lam} | \mathbf X |$
  be the expected density of the hard sphere model on $\Lam$ with
  free boundary conditions, and let
  \begin{equation*}
    F_{\Lam}(\lam) = \frac{ \E_{\Lam,\lam}  | \{ y \in \Lam_{\mathrm{Int}} : \mathrm{dist}(y,
      \mathbf X) >2r \}| }{ |\Lam|} 
  \end{equation*}
  be the expected \textit{free volume} fraction of $\Lam$.  A
  short calculation gives the identity
  $\rho_{\Lam}(\lam) = \lam F_{\Lam}(\lam)$ for all bounded measurable
  $\Lam$ of positive volume~\cite[Lemma 7]{jenssen2019hard}.  Further,
  \begin{equation*}
    F_{\Lam}(\lam) \ge |\Lam_{\mathrm{Int}}|/|\Lam|- 2^d \rho_{\Lam}(\lam) \,,  
  \end{equation*}
  since each center in $\mathbf X$ can block at most volume
  $2^d$. With a little algebra, this implies
  \begin{equation*}
 \rho_{\Lam}(\lam) \geq  \frac{|\Lam_{\mathrm{Int}}|  }{|\Lam|   }\frac{\lam}{1+\lam 2^{d}}.  
  \end{equation*}
Applying this bound to $\Lam = Q_n$ and taking a limit gives the lemma.
\end{proof}

We will also require the following bound on $\rho(\lam)$.
\begin{theorem}[{\cite[Proof of Theorem~2]{jenssen2019hard}}]
\label{thmJJP}
For all $d \ge 2$ and all $\lam >0$, 
\[ \rho(\lam) \ge \inf_z \max \left \{ \lam e^{-z},  z 2^{-d} e^{-2\lam 3^{d/2}}  \right \}  \,. \]
In particular if $\lam = c 2^{-d}$, we have $\rho(\lam) \ge (1+o_d(1)) W(c) 2^{-d}$ where $W(\cdot)$ is the Lambert-W function, i.e. the inverse of $f(W) =  W e^W$. 
\end{theorem}

\begin{proof}[Proof of Theorem~\ref{thmDensityBounds}]
  To prove the first statement in Theorem~\ref{thmDensityBounds}, we
  combine Lemma~\ref{lemEasybound} and Theorem~\ref{thmGenThm} to get
  \begin{equation*}
    \rho_c \ge \frac{2^{1-d}}{1+ 2^d 2^{1-d}} =  \frac{2}{3 \cdot 2^d}   \, ,
  \end{equation*}
  
  To prove the second statement in Theorem~\ref{thmDensityBounds}, we
  use Theorem~\ref{thmJJP} and the bound  $\lam_c(d) \ge 2^{1-d}$, to obtain
  \begin{equation*}
    \rho_c(d) \ge (1+o_d(1))W(2) 2^{-d} \ge ( .8526+o_d(1))   \cdot 2^{-d}\,, 
  \end{equation*}
  as $d \to \infty$.
\end{proof}

\section*{Acknowledgements}
We thank Arnaud Marsiglietti for pointing us towards
Lemma~\ref{lemParallelVolume}.  We thank the anonymous referees for many helpful comments and suggestions.  TH supported in part by EPSRC grant EP/P003656/1. WP supported in part by NSF Career award DMS-1847451.  SP supported in part by NSF Graduate
Research Fellowship DGE-1650044.

\end{document}